\documentclass{amsart}

\title{Smallest non-cyclic quotients of braid and mapping class groups}
\author{Sudipta Kolay}
\address{School of Mathematics \\ Georgia Institute of Technology\\Atlanta, GA 30332, USA}
\email{skolay3@gatech.edu}
\date{}

\usepackage[a4paper,top=3cm,bottom=2cm,left=3cm,right=3cm,marginparwidth=1.75cm]{geometry}

\usepackage{amsmath,amsthm,amsfonts,amssymb}
\usepackage{graphicx}
\usepackage[colorlinks=true, allcolors=blue]{hyperref}
\usepackage{enumitem} 
\usepackage{dirtytalk}

\theoremstyle{plain}
\newtheorem{thm}{Theorem}
\newtheorem{cor}[thm]{Corollary}

\newtheorem{lem}[thm]{Lemma}
\newtheorem{claim}[thm]{Claim}

\theoremstyle{definition}

\theoremstyle{remark}
\newtheorem{rem}[thm]{Remark}

\begin{document}

\begin{abstract}
    We show that the smallest non-cyclic quotients of braid groups are symmetric groups, proving a conjecture of Margalit. Moreover we recover results of Artin and Lin about the classification of homomorphisms from braid groups on $n$ strands to symmetric groups on $k$ letters, where $k$ is at most $n$. Unlike the original proofs, our method does not use the Bertrand-Chebyshev theorem, answering a question of Artin. Similarly for mapping class group of closed orientable surfaces,  the smallest non-cyclic quotient is given by the mod two reduction of the symplectic representation. We 
    provide an elementary proof of this result, originally due to  Kielak--Pierro, which proves a conjecture of Zimmermann.
\end{abstract}

\maketitle

\section{Introduction}
The goal of this paper is to show that, with some obvious exceptions, the smallest non-cyclic quotients of the braid  and mapping class groups, are given by natural projections\footnote{See the next section for details.} $\pi:B_n\rightarrow S_n$ (forgetful map) and $\Phi:\mathrm{Mod}(\Sigma_{g})\rightarrow \mathrm{Sp}(2g,\mathbb{Z}_2)$ (mod two reduction of the symplectic representation). We begin by stating our main result for the Artin braid groups $B_n$.

\begin{thm}\label{A} Suppose $n=3$ or $n\geq 5$. If $G$ is a non-cyclic quotient of $B_n$, then either $|G|> |S_n|=n!$, or $G$ is isomorphic to $S_n$. Moreover, in the latter case the quotient map $B_n\rightarrow G$ is obtained by postcomposing the natural map $\pi$ with an automorphism of $S_n$.
\end{thm}
\noindent 

There are no non-cyclic quotients of $B_n$ for $n\leq 2$, and for $n=4$  the smallest non-cyclic quotient is $S_3$, which is proved in Claim~\ref{c34} of Section~\ref{back}. Hence the hypothesis $n=3$ or $n\geq 5$ is necessary in the theorem above.

The first statement of this theorem proves a conjecture of Margalit \cite{chudnovsky2020finite,scherich2020finite}, stating that the smallest non-cyclic quotient of $B_n$ is $S_n$ for $n\geq 5$. For the non-trivial cases $n \in \{5, 6\}$, this was first proved by Caplinger-Kordek~\cite{caplinger2020small}, and several  recent papers \cite{chudnovsky2020finite,caplinger2020small,scherich2020finite} prove lower bounds for the order of non-cyclic quotients of braid groups, using totally symmetric sets~\cite[Section 2]{first_TSS}, towards proving Margalit's conjecture. Our work builds further upon the idea of totally symmetric sets; see the discussion after Lemma~\ref{lemA}.

Since the automorphisms of symmetric groups are well understood, the second statement in the theorem above immediately implies for $n\neq 4$, the characterization of non-cyclic\footnote{By this we mean the image of the homomorphism is not cyclic.} homomorphisms from $B_n\rightarrow S_k$, with $k\leq n$, originally due to Artin~\cite{artinbrper} for  $k=n$ (and transitive homomorphisms) and improved by Lin~\cite[Theorem 3.9]{lin2004braids} for the remaining cases.
\begin{cor}\label{ALc}
For $n\geq 3$, and $n\neq 4,6$, all non-cyclic homomorphisms  $f:B_n \rightarrow S_n$ are conjugate to the standard projection $\pi$. Also, the only exceptional (up to conjugation) homomorphism $f:B_6\rightarrow S_6$ comes from composing $\pi$ with the only non-trivial (up to conjugation) outer automorphism of $S_6$ defined by
$ (12)\mapsto (1,2)(3,4)(5,6)$ and $(1,2,3,4,5,6) \mapsto (1,2,3)(4,5)$.

\end{cor}

 Artin noted that his proof in \cite{artinbrper}
\say{uses the existence of a prime between $\frac{n}{2}$ and $n - 2$ for $n > 7$ but it would be preferable if a proof could be found that does not make use of this fact}. This fact, known as the Bertrand--Chebyshev theorem~\cite{Chebyshevbp}, is also crucial for Lin's proof of the above result ~\cite[Theorem 3.9]{lin2004braids}. Our proof here does not use this fact (and to the best of our knowledge, this is the first such proof).

\begin{rem}[Exceptional case $n=4$, Artin~\cite{artinbrper}]\label{re4} For completeness, we will record here the exceptional non-cyclic homomorphisms (up to conjugations) from $B_4$ to $S_k$ with $k\leq 4$. Let $\sigma_1, \sigma_2, \sigma_3$ denote the Artin generators of $B_4$, and let $\alpha=\sigma_3\sigma_2\sigma_1$. We see that $B_4$ is generated by $\sigma_1$ and $\alpha$.

\begin{enumerate}
    \item $f_1:B_4\rightarrow S_4$ defined by $\sigma_1\mapsto (1,2,3,4)$, $\alpha\mapsto (1,2)$.
    \item $f_2:B_4\rightarrow S_4$ defined by $\sigma_1\mapsto (1,3,2,4)$, $\alpha\mapsto (1,2,3,4)$.
       \item $f_3:B_4\rightarrow A_4 \subset S_4$ defined by $\sigma_1\mapsto (1,2,3)$, $\alpha\mapsto (1,2)(3,4)$. (Here $A_4$ denotes alternating group on four letters, which uniquely embeds in $S_4$)
         \item $f_4:B_4\rightarrow S_3 (\subset S_4$) defined by $\sigma_1\mapsto (1,2)$, $\alpha\mapsto (1,3)$.
\end{enumerate}
\end{rem}

Our main result for mapping class groups $\mathrm{Mod}(\Sigma_{g})$ of closed orientable surfaces parallels Theorem~\ref{A}, and is essentially the same as the result of Kielak--Pierro \cite{Kielak2017OnTS}, using other methods.

\begin{thm}  \label{C}
Let $g\geq 1$. For any non-cyclic quotient $H$ of $\mathrm{Mod}(\Sigma_{g})$, either $|H|>|\mathrm{Sp}(2g,\mathbb{Z}_2)|$, or $H$ is isomorphic to $\mathrm{Sp}(2g,\mathbb{Z}_2)$. Moreover, in the latter case the quotient map $\mathrm{Mod}(\Sigma_{g})\rightarrow H$ is obtained by postcomposing $\Phi$ with an automorphism of $\mathrm{Sp}(2g,\mathbb{Z}_2)$.
\end{thm}

Zimmermann~\cite{Zimmermann2008ANO} proved that for $g\in \{3,4\}$, the smallest non-trivial\footnote{ For $g\geq 3$, $\mathrm{Mod}(\Sigma_{g})$ is perfect and therefore its smallest {non-trivial} and {non-cyclic} quotients are the same.} quotient of $\mathrm{Mod}(\Sigma_{g})$ is $\mathrm{Sp}(2g,\mathbb{Z}_g)$, and conjectured the same statement holds for arbitrary $g\geq 3$. This conjecture was first proved by Kielak--Pierro \cite{Kielak2017OnTS} using the classification of finite simple groups and representation theory of mapping class groups. Moreover, Kielak--Pierro proved the same result holds for quotients of $\mathrm{Mod}(\Sigma_{g}^b)$ where $b$ is the number of boundary components, and we further extend their result here by allowing punctures as well.

\begin{thm}\label{Cb}
 Let $g\geq 3$. The smallest non-trivial quotient of $\mathrm{Mod}(\Sigma_{g,n}^b)$ is $\mathrm{Sp}(2g,\mathbb{Z}_2)$ for $n\in\{0,1\}$, and $\mathbb{Z}_2$  for $n\geq 2$.  If we furthermore assume $n\geq 5$, any non-cyclic quotient of $\mathrm{Mod}(\Sigma_{g,n}^b)$ of smallest order is isomorphic to either $S_n$ or  $\mathrm{Sp}(2g,\mathbb{Z}_2)$ (depending on which group is smaller). Moreover, in any of the above cases, any epimorphism to a quotient of smallest order is the standard projection, postcomposed with an automorphism of the image.

\end{thm}

As indicated, some of the results above were previously known, but our proofs are considerably easier. For example, we do not use the classification of finite simple groups or the Bertrand--Chebyshev theorem. We use an \emph{inductive orbit stabilizer method}, described in Section~\ref{IOS}, which should also be applicable in other settings. Our approach is similar to that of  Chudnovsky--Kordek--Li--Partin \cite{chudnovsky2020finite}, 
Caplinger--Kordek \cite{caplinger2020small}, and particularly Scherich--Verberne \cite{scherich2020finite}; in that we all consider some group actions of the quotient (of braid groups), and use the orbit stabilizer theorem to find a bound on the size of the quotient. The advantage of our approach is that we prove an optimal lower bound on orbit size (by looking at the corresponding orbit size in the candidate smallest quotient), and moreover using induction to find the stabilizer size. For the two families of groups we consider here, this not only gives us the optimal lower bounds for size of the smallest quotient at the numerical level, but we also obtain the smallest quotient group up to isomorphism, and moreover a characterization of all possible minimal quotient maps.

Let us note that if $G\rightarrow H$ and $H\rightarrow I$ are surjective group homomorphisms, and $I$ is smallest non-cyclic (respectively non-trivial) quotient of $G$, then $I$ is also the smallest non-cyclic (respectively non-trivial) quotient of $H$. Thus, an immediate consequence of Theorems~\ref{C} and ~\ref{Cb}  is the following result:
\begin{cor}
For  $g\geq 1$ (respectively for $g\geq 3$), $\mathrm{Sp}(2g,\mathbb{Z}_2)$ is the smallest non-cyclic (respectively non-trivial) quotient of $\mathrm{Sp}(2g,\mathbb{Z})$.
\end{cor}

\noindent \textit{Acknowledgements}. The author would like to thank Dan Margalit for various useful discussions, suggesting to look at results for mapping class groups, and especially for explaining to us the much shorter proof of Lemma~\ref{lemA}. The author is grateful to John Etnyre for helpful suggestions. The author thanks the referee for comments and corrections. The author is grateful to Dawid Kielak and Emilio Pierro for comments on an earlier draft of this paper. This work is partially supported by NSF grant DMS-1906414.

\section{Background}\label{back}

In this section we will collect several necessary definitions and results. We will also prove a claim, which will serve as base cases for inductive proofs later.

\subsection*{Braid Groups}
The most well known quotient of the braid group $B_n$ \cite{artin1947br2} on $n$ strands is the symmetric group $S_n$ on $n$ letters, obtained by forgetting all crossing information. This quotient map $\pi:B_n\rightarrow S_n$ can alternately be described as adding the relations $\sigma_i^2=1$ (here $\sigma_i$ are half twists) to  the Artin presentation~\cite{artin1947br2} of the braid group $B_n$:
$$B_n=\{\sigma_1,...,\sigma_{n-1}| \sigma_i\sigma_{i+1}\sigma_i=\sigma_{i+1}\sigma_i\sigma_{i+1} \text{ for all } 1\leq i < n-1, \sigma_i\sigma_j=\sigma_j\sigma_i \text{ if } |i-j|>1\}. $$
\emph{Birman-Ko-Lee generators \cite{BIRMAN1998322}}: Consider $B_{n}$ as the mapping class group of the closed unit disc with $n$ marked points $p_1,...,p_{n}$ with increasing first co-ordinates and identical second coordinate. 

Consider for all $1\leq i\neq j \leq n$, the arcs $\gamma_{i,j}=\gamma_{j,i}$ joining the $p_i$ and $p_j$ going over all $p_k$ between $p_i$ and $p_j$, and let $\rho_{i,j}$ denote the right handed half twists about $\gamma_{i,j}$. For $1\leq i<j\leq n $ the various $\rho_{i,j}$ are the Birman-Ko-Lee generators of the braid group $B_n$, and we note that $\sigma_i=\rho_{i,i+1}$.

\subsection*{Mapping Class Groups}
Let $\Sigma_{g,n}^b$ denote the orientable surface of genus $g$, with $n$ punctures and $b$ boundary components (where we will drop $n$ and $b$ from the notation if they are zero), and we will denote its mapping class group by $\mathrm{Mod}(\Sigma_{g,n}^b)$. Our convention is that mapping classes preserve orientation, fixes boundary components, and can permute the punctures. The subgroup $\mathrm{PMod}(\Sigma_{g,n}^b)$ will denote the pure mapping class group, consisting of mapping classes that fixes the punctures.

We get an epimorphism $\Phi$ from $\mathrm{Mod}(\Sigma_{g}^b)$  by composing the capping homomorphism~\cite[Section 3.6.2]{farb2012primer} with the symplectic representation~\cite[Section 6.3]{farb2012primer}, and the mod two reduction.
   $$ \mathrm{Mod}(\Sigma_{g}^b)\xrightarrow{capping} \mathrm{Mod}(\Sigma_{g}) \xrightarrow {symplectic} \mathrm{Sp}(2g,\mathbb{Z})\xrightarrow {reduce} \mathrm{Sp}(2g,\mathbb{Z}_2). $$
More generally, for $\Sigma_{g,n}$, let us consider the action of mapping class group on homology. If we take a free basis of $H_1(\Sigma_{g,n},\mathbb{Z})$ by taking standard symplectic basis curves for each genus, and a the class of a loop surrounding each puncture, the action of any mapping class can be represented by an invertible integral matrix in $\mathrm{GL}(2g+n,\mathbb{Z})$. Moreover for any such matrix, the top left block is symplectic matrix, the top right block is zero, and the bottom right block will be a permutation matrix.
Thus by projecting to diagonal blocks, we obtain epimorphisms from $\mathrm{Mod}(\Sigma_{g,n})$ (and thus from $\mathrm{Mod}(\Sigma_{g,n}^b)$ as well by capping) to $\mathrm{Sp}(2g,\mathbb{Z})$ (and hence to $\mathrm{Sp}(2g,\mathbb{Z}_2)$) and $S_n$. We will call these homomorphisms to be standard projections from $\mathrm{Mod}(\Sigma_{g,n}^b)$ to $\mathrm{Sp}(2g,\mathbb{Z}_2)$ and $S_n$. It can be seen that this standard projection from $\mathrm{Mod}(\Sigma_{g,n}^b)$ to $S_n$ is the same  as the induced action of the mapping classes on the punctures.\\

\noindent\emph{Some facts about symmetric and symplectic groups}: It is well known that for $n\geq 5$, the only non-trivial quotient of $S_n$ is  $\mathbb{Z}_2$ (obtained by modding out by the simple group $A_n$). Also, it is known~\cite[Chapter 3]{Grove2001ClassicalGA} that the symplectic group $\mathrm{Sp}(2g,\mathbb{Z}_2)$ is simple for $g\geq 3$, and for the exceptional cases we have the isomorphisms $\mathrm{Sp}(2,\mathbb{Z}_2)\cong S_3$ and $\mathrm{Sp}(4,\mathbb{Z}_2)\cong S_6$.\\

The following claim gives the base cases for our inductive proofs later. 
\begin{claim}\label{c34}
The smallest non-cyclic quotient of $B_3$, $B_4$ and $\mathrm{Mod}(\Sigma_{1})=\mathrm{SL}(2,\mathbb{Z})$ is $S_3$. Moreover all epimorphisms from these three groups to $S_3$ are related by a conjugation of $S_3$.

\end{claim}
\begin{proof}
 
 The natural homomorphism $\pi$, and $f_4$ in Remark~\ref{re4} shows $S_3$ is a quotient of $B_3$ and $B_4$ respectively. Moreover, it is easy to see that $\pi:B_3\rightarrow S_3$ factors through  $B_3/Z(B_3)\cong \mathrm{PSL}(2,\mathbb{Z})$, and thus $S_3$ is a quotient of $\mathrm{PSL}(2,\mathbb{Z})$ and hence $\mathrm{SL}(2,\mathbb{Z})$.
 We note that the all groups except $S_3$ of order at most $|S_3|=6$  are abelian (the only non-cyclic group among them is the Klein four group), and thus cannot be a non-cyclic quotient of a group with cyclic abelianization (such as braid groups or $\mathrm{SL}(2,\mathbb{Z})$).
  The last statement of the claim follows\footnote{ For $n=4$, a similar (but more tedious) check verifies Remark~\ref{re4}.} by noting that the only pair of non-commuting elements in $S_3$ satisfying the braid relation are the transpositions.\end{proof}

\section{The inductive orbit stabilizer method}\label{IOS}

The orbit stabilizer theorem is widely used in computing orders of finite groups which naturally act on a space, and as this paper illustrates, it is also useful for determining orders of smallest non-cyclic\footnote{It may be possible to adapt this method to find smallest non-trivial/non-abelian/non-solvable quotient.} quotients of groups. In our context we work with an infinite family of groups, and we can use the orbit stabilizer theorem inductively. We formulate the steps of the method below. While this method may not be new, proofs of similar results in the literature seem to rely on more complicated methods, as mentioned in the Introduction.

Suppose we have a nested family of groups $(G_n)_{n\geq 1}$ with cyclic abelianizatons.  If we want to show the smallest non-cyclic quotient is the family of groups $(H_n)_{n\geq 1}$, with a family of quotient maps $\pi_n:G_n \rightarrow H_n$, it suffices to carry out the following steps (after checking base cases).

\begin{enumerate}
    \item \emph{Lower bound on orbit size}:  Find the size $k$ of an orbit of the conjugation action of $H_n$. Find a suitable collection collection of elements $x_1,...,x_k$ in $G_n$ so that their images generate the orbit, and show that the normal closure of each $x_ix_j^{-1}$ contains the commutator subgroup $G'_n$ of $G_n$ (equivalently, under any non-cyclic quotient of $G_n$, the quotient classes $\overline{x_i}$ are all distinct).

    \item \emph{Inductively find size of stabilizer}: For some non-cyclic quotient $q:G_n\rightarrow I_n$, inductively bound the size of the stabilizer of the quotient class of $q(x_1)$ in $I_n$, so that the orbit stabilizer theorem implies $|I_n|\geq |H_n|$. For instance, if the centralizer of $x_1$ contains $\langle x_1 \rangle\times G_{n-i}$, it may be possible to get the desired result by applying the inductive hypothesis on the induced quotient $G_{n-i}\rightarrow q(G_{n-i})/Z(q(G_{n-i}))$.
    Finally, if $|I_n|=|H_n|$, show that $I_n$ is isomorphic to $H_n$. This follows if the kernel of $q$ contains the kernel of $\pi_n$, which moreover shows any epimorphism from $G_n$ to $H_n$ is $\pi_n$ composed with an automorphism of $H_n$.
\end{enumerate}

Some modifications, such as considering a different group action, may be needed to make this method work in a particular situation, and we will see one such modification for the mapping class groups case later.

\section{Smallest non-cyclic quotients of braid groups}

We will carry out the steps of the inductive orbit stabilizer method here for Artin braid groups, and show that smallest non-cyclic quotients are symmetric groups.

  \subsection*{Lower bounds for size of orbit} Let us begin by observing that the conjugacy class of all transpositions in $S_n$ consists of $n\choose 2$ elements. We will take $x_i$'s to be the Birman-Ko-Lee generators  of the braid group, as mentioned in Section~\ref{back}.
 The following lemma will complete the first step.

\begin{lem}\label{lemA} For $n\geq 5$, and a non-cyclic quotient of $B_{n}$, the $n \choose 2$ quotient classes $\overline{\rho_{i,j}}$ are distinct.

\end{lem}

We should note that the lemma does not hold for $n=4$, as there is an exceptional homomorphism from $B_4$ to $B_3$ (which can be further quotiented to obtain $f_4:B_4\rightarrow S_3$ mentioned in Remark~\ref{re4}) defined by $\sigma_1\mapsto \sigma_1$, $\sigma_2\mapsto \sigma_2$ and $\sigma_3\mapsto \sigma_1$.

Totally symmetric sets are subsets of a
group with the property that any homomorphism restricts to an injective
map on that set or to a trivial map on that set (that is not the
definition, but a consequence, see ~\cite[Lemma 2.1]{first_TSS}).  Lemma~\ref{lemA} can be similarly phrased as
saying that the set $\{\rho_{i,j}\}$ satisfies this same property. We will give two proofs of this lemma, the first is essentially in \cite[Lemma 4.2]{chen2019homomorphisms}, and the second is more hands-on.

\begin{proof} Suppose we have $\gamma_{i,j}$ and $\gamma_{k,l}$ with $\{i,j\}\neq \{k,l\}$ having the same quotient class. Since $n\geq 5$, we can find an arc $\delta$ between two marked points disjoint from $\gamma_{k,l}$, and sharing an endpoint with $\gamma_{i,j}$. It follows that $\delta$ and its image under $\rho_{i,j}\rho_{k,l}^{-1}$  shares one endpoint and have disjoint interiors. Thus by a change of coordinates principle \cite[Section 1.3.2]{farb2012primer}, the commutator of $\rho_\delta$ (the right handed half twist about $\delta$) and $\rho_{i,j}\rho_{k,l}^{-1}$ is conjugate to $\sigma_1\sigma_2^{-1}$. Now, as $\rho_{i,j}\rho_{k,l}^{-1}$ is in the kernel of the quotient map, so is its commutator with $\rho_\delta$, and thus so is $\sigma_1\sigma_2^{-1}$. The result now follows since $\sigma_1\sigma_2^{-1}$ normally generates $B'_n$ (which is a direct consequence of the braid and far commutation relations), and the fact that $B_n/B'_n$ is cyclic.
\end{proof}

\begin{proof}[Alternate proof] We will repeatedly use the following two observations:
\begin{enumerate}
    \item If two elements $x,y$ in any group satisfy both braid and far commutation relations, then $xyx=yxy \Rightarrow xyx=xyy \Rightarrow x=y$, i.e. $x$ and $y$ must coincide.
    \item For any distinct $i,j,k$, if $\overline{\rho_{i,j}}$ is same as $\overline{\rho_{j,k}}$, then by the partial commutation relation\footnote{for an appropriate $\epsilon\in\{-1,1\}$  depending on the relative position among $i,j,k$, we have $\rho_{i,j}^\epsilon (\gamma_{j,k})=\gamma_{i,k}$, and hence we get partial commutation relation $\rho_{i,k}=\rho_{j,k}^{\epsilon} \rho_{i,j}\rho_{j,k}^{-\epsilon}$.}, they are also equal to $\overline{\rho_{i,k}}$.
\end{enumerate}

Now, let us suppose the lemma is not true, let us first consider the case we have $\overline{\rho_{i,j}}=\overline{\rho_{j,k}}$ with distinct $i,j, k$, and by the second observation above, we may assume $i<j<k$. For any $l$ distinct from $i,j,k$ we see that
  if $l$ is (respectively is not) between $i$ and $j$, by the first observation we have   $\overline{\rho_{k,l}}=\overline{\rho_{j,k}}$ (respectively  $\overline{\rho_{k,l}}=\overline{\rho_{i,j}}$).
 By repeatedly applying the second observation, we see all the $\overline{\rho_{i,j}}$'s must coincide, and thus the quotient is cyclic (as $B_{n}$ is generated by the half twists $\sigma_i$'s), a contradiction.

Let us now consider the case we have $\overline{\rho_{i,j}}=\overline{\rho_{k,l}}$ for distinct $i,j,k,l$. 
Since $n+1\geq 5$, we can find $m$ distinct from all $i,j,k,l$.
Let $o\in \{i,j,k,l\}$ be such that $|o-m|$ is smallest. By symmetry, without loss of generality we may assume that $o\in\{i,j\}$. We see that $\overline{\rho_{i,j}}$ and $\overline{\rho_{o,m}}$ satisfies both braid relation (as $o$ is common) and far commutation relation (as $\overline{\rho_{i,j}}=\overline{\rho_{k,l}}$, and  $\gamma_{k,l}$ and $\gamma_{o,m}$ are disjoint). By the first observation, we must have $\overline{\rho_{i,j}}=\overline{\rho_{o,m}}$, and by our discussion in the previous paragraph, all the $\overline{\rho_{i,j}}$'s must be the same, again leading to a contradiction.
\end{proof}

 \subsection*{Inductive step} 
  Now we will use induction to prove the Theorem~\ref{A} (and we repeat the statement below for convenience).\\
  \noindent\textit{Inductive hypothesis:} Suppose $n=3$ or $n\geq 5$. If $G$ is a non-cyclic quotient of $B_n$, then either $|G|> |S_n|=n!$, or $G$ is isomorphic to $S_n$.  Moreover, in the latter case the quotient map $B_n\rightarrow G$ is obtained by postcomposing the natural map $\pi$ with an automorphism of $S_n$.

 We will use induction on $n$ in steps of two, and we will use the base case $n=3$ from Claim~\ref{c34}, and the base case $n=6$ from the computer assisted proof Caplinger-Kordek~\cite{caplinger2020small}. But we can also do the $n=6$ case by hand with a separate argument similar to the inductive proof, as explained after this proof.
 
 \noindent\textit{Proof idea:} Let us note that the centralizer of a transposition $(1,2)$ in $S_n$ is $\{1,(1,2)\}\times S_{n-2}$, where $S_{n-2}$ is the symmetric group on the letters $3,...,n$. Similarly, we see that the centralizer of $x_1=\sigma_1$ in $B_n$ contains $\langle \sigma_1 \rangle \times B_{n-2}$, which projects to $\{1,(1,2)\}\times S_{n-2}$ under $\pi$. If under some non-cyclic quotient of $B_n$, the centralizer of $\overline{x_1}$ is  $\langle\overline{x_1}\rangle\times \overline{B_{n-2}}$, then use inductive hypothesis on the size of $\overline{B_{n-2}}$. But, $\langle\overline{x_1}\rangle$ and $\overline{B_{n-2}}$ may not intersect trivially, however we see that their intersection is central in $\overline{B_{n-2}}$. Therefore, we can use the inductive hypothesis on $\overline{B_{n-2}}/Z(\overline{B_{n-2}})$.

\begin{proof} [Proof of Theorem~\ref{A}]
As mentioned above, we will use the base cases $n\in\{3,6\}$, and since use induction on $n$ in steps of two, and and this will imply the result for all odd $n\geq 5$, and even $n\geq 8$.

We will assume the inductive hypothesis is true for $k=n-1$ and prove the statement for $k=n+1$ (with $n+1\geq 5)$.
Suppose $q:B_{n+1} \rightarrow G$ be a non-cyclic quotient of smallest order. By Lemma~\ref{lemA}, it follows that all the $\frac{(n+1)n}{2}$ quotient classes $\overline{\rho_{i,j}}$'s must be distinct for non-cyclic $G$. It is known that all the $\rho_{i,j}$'s are conjugate in $B_{n+1}$, so $\overline{\rho_{i,j}}$'s are conjugate in $G$. Therefore, if we consider the group action of $G$ on itself by conjugation, the orbit stabilizer theorem tells us
\begin{equation}\label{eq1}
   |G|=|O||C| \geq \frac{(n+1)n}{2}|C|,
\end{equation}
 where $C$ denotes the centralizer (i.e. stabilizer of the conjugation action) of the element $\overline{\rho_{1,2}}$, and $O$ denotes its conjugacy class (i.e. the image of the half twists). Since $\sigma_1=\rho_{1,2}$ commutes with the subgroup $V_{1,2}$ of $B_{n+1}$ generated by $\sigma_3,...,\sigma_n$ (thus $V_{1,2}$ is isomorphic to $B_{n-1}$), we see $C$ contains $H_{1,2}:=q(V_{1,2})$ as a subgroup, and clearly it also contains $\overline{\rho_{1,2}}$.
It follows from Lemma~\ref{lemA} it $H_{1,2}$ is not cyclic, and so we can apply inductive hypothesis to any non-cyclic quotient of $H_{1,2}$.

Let $M$ denote the cyclic subgroup  generated by $\overline{\rho_{1,2}}$ in $G$. We see that  $Y=H_{1,2}\cap M$ is in the center $Z$ of $H_{1,2}$ as $\overline{\rho_{1,2}}$ commutes with all elements $H_{1,2}$. 
 In case $H_{1,2}/Z$ is cyclic, we know that $H_{1,2}$ is abelian, but as $H_{1,2}$ is a quotient of $V_{1,2}\cong B_{n-1}$, it has to factor through the abelianization and therefore cyclic, contradicting Lemma~\ref{lemA}. Hence $H_{1,2}/Z$ is a non-cyclic quotient of $B_{n-1}$, and so by the inductive hypothesis for $k=n-1$, we have $|H_{1,2}/Z|\geq (n-1)!$. Thus we have $|H_{1,2}|\geq |Z| (n-1)! \geq |Y| (n-1)!$. Also if $D$ denotes the subgroup of $C$ generated by $H_{1,2}$ and $M$, we see that $M$ is in the center of $D$ and thus $|D|=|M/Y||H_{1,2}|\geq|M|(n-1)!$

By combining with Equation~\eqref{eq1}, we see that \begin{equation}
    |G|\geq \frac{(n+1)n}{2}|C|\geq \frac{(n+1)n}{2}|D|\geq \frac{(n+1)n}{2}|M| (n-1)!=(n+1)!\frac{|M|}{2}.
\end{equation}
Thus the only way $|G|\leq (n+1)!$ is if $|M|=1$ (in this case $\overline{\rho_{1,2}}=1$ so $G$ is the trivial group, contradiction) or $|M|=2$. If the latter case happens then $q(\sigma_i^2)=1$ for all $i$, and thus $q$ factors through the standard quotient map $\pi: B_{n+1}\rightarrow S_{n+1}$. Since the only proper quotient of $S_{n+1}$ (for $n+1\geq 5$) is $\mathbb{Z}_2$, it must be the case that $G$ is isomorphic to $S_{n+1}$, as required. Moreover this shows that $q$ is a composition of the standard map $\pi$ with an automorphism of $S_{n+1}$. \end{proof}

\begin{proof}[Proof of Theorem~\ref{A} for n=6]

We will show the desired result for this case using a similar argument as above, and we use the same notation. Let $m$ denote the order of  $\overline{\rho_{1,2}}$ in $G$ (a non-cyclic quotient of $B_6$ of smallest order). If $m=2$, we know $q:B_6\rightarrow G$ factors through $S_6$, and therefore the desired result holds, so we will assume $m>2$ hereafter. By Equation \eqref{A}, we have $|G|\geq 15|C|\geq 15|H_{1,2}|$. The following claim gives a lower bound on $|H_{1,2}|$ which implies $|G|\geq 6!$, and thus $|G|=6!$. 

\begin{claim}  For $m>2$, we have $|H_{1,2}|\geq 48$, and equality holds only if $m=4$ and $\overline{\sigma_3}^2=\overline{\sigma_5}^2$.
\end{claim}
\begin{proof}
We see that the ${4 \choose 2}=6$ elements $\overline{\rho_{i,j}}$ are distinct for $3\leq i <j \leq 6$ (we are applying Lemma~\ref{lemA} for $n=6$, and not 4). Thus by the orbit stabilizer theorem we have $|V_{1,2}|=|\hat{O}| |\hat{C}|$, where $\hat{O}$ and $\hat{C}$ denotes the orbit and centralizer of the element $\overline{\rho_{3,4}}=\overline{\sigma_3}$ in $V_{1,2}$.  We see that $\hat{C}$ contains the cyclic subgroups generated by the commuting elements $\overline{\sigma_3}$ and $\overline{\sigma_5}$.

In case these subgroups coincide, we will have $\overline{\sigma_5}=\overline{\sigma_3}^p$ for some $p$, and by an appropriate conjugations in $G$ (by the image of a periodic braids), we get $\overline{\sigma_3}=\overline{\sigma_1}^p$ , and $\overline{\sigma_4}=\overline{\sigma_2}^p$. It would therefore follow that $G$ is generated by $\overline{\sigma_1}$ and $\overline{\sigma_2}$, but then the stabilizer of $\overline{\sigma_4}$ is all of $G$, contradicting that we have a non-trivial orbit of $\overline{\sigma_4}$. Thus $\hat{C}$ properly contains the cyclic subgroup generated by $\overline{\sigma_5}$, and so $|\hat{C}|\geq 2m$.
For $m=3$, we see the subgroups generated by $\overline{\sigma_3}$ and $\overline{\sigma_5}$ cannot intersect (or otherwise they coincide) and therefore  $|\hat{C}|= 9$, and thus $|V_{1,2}|\geq 6\cdot 9=54$. Lastly, for $m\geq 4$, we have $|\hat{C}|\geq 2m$
 and so $|V_{1,2}|=|\hat{O}| |\hat{C}|\geq 6\cdot 2m=12 m\geq 48$. Moreover it is easily checked that $|V_{1,2}|=48$ if and only if $m= 4$ and $\overline{\sigma_3}^2=\overline{\sigma_5}^2$. \end{proof}

 It remains to consider the case $|G|= 6!$, $m=4$ and $\overline{\sigma_3}^2=\overline{\sigma_5}^2$. By conjugation by image of a periodic braid it follows $\overline{\sigma_1}^2=\overline{\sigma_3}^2$. The non-trivial (since $m\neq 2$) element $\overline{\sigma_1}^2$ (commuting with $\overline{\sigma_1},\overline{\sigma_3},\overline{\sigma_4},\overline{\sigma_5}$) is in the center of $G$ , as $\overline{\sigma_2}$ commutes with $\overline{\sigma_5}^2$ ($=\overline{\sigma_1}^2$). Thus $G$ has non-trivial center $Z(G)$, and so $G/Z(G)$ must be a strictly smaller non-cyclic quotient of $B_6$, a contradiction.
\end{proof}

We will now see how Theorem~\ref{A} implies Artin and Lin's results.
\begin{proof}[Proof of Corollary~\ref{ALc}]
If $f:B_n\rightarrow S_k$ is a non-cyclic homomorphism, by Theorem~\ref{A}, we must have $k=n$ and  we have $f=g\circ \pi$, where $g:S_n\rightarrow S_n$ is an automorphism. Now we use the fact,  due to H\"{o}lder \cite{Hol}, that for $n\neq 2,6$ all automorphisms of $S_n$ are inner, and there is exactly one outer automorphism of $S_6$ up to conjugation, which is mentioned in the statement of the corollary.
\end{proof}

\section{Smallest non-cyclic quotients of mapping class groups }
We will use a slightly modified form of the inductive orbit stabilzer method here. In the inductive step, it will be more convenient to look at the conjugation action on a pair (instead of a single element) of elements of the quotient.
  \subsection*{Lower bounds for size of orbit} We note that the orbit of all transvections in $\mathrm{Sp}(2g,\mathbb{Z}_2)$ is $2^{2g} -1$, since these are in bijection with primitive vectors in $(\mathbb{Z}/{2\mathbb{Z})^{2g}}$. In this case we will take $x_i$'s to be suitable right handed Dehn twists about simple closed curves so that their mod two homology classes give us all primitive vectors in $(\mathbb{Z}/{2\mathbb{Z})^{2g}}$. Corresponding to each primitive vector $v$ with 0's and 1's in first homology $H_1(\Sigma_g,\mathbb{Z})$, we will construct simple closed curve $\alpha_v$ realizing this homology class, and denote by $T_v$ the right handed Dehn twist about the curve $\alpha_v$. Starting at the first entry of $v$, for each non-zero pair $(p,q)$ of entries we can draw the $(p,q)$ curve on the corresponding genus, and we can join these curves by standard bands running straight across. For instance, the red and green curves in the leftmost picture in Figure~\ref{orb} shows the $(1,0)$ and $(0,1)$ curves on a genus, which is then band-summed with the other $(p,q)$ curves. It is easy to see that if we have two binary vectors $v$ and $w$, which differ on the same pair of entries, then by localizing to the corresponding genus, we can find a third simple closed curve $\beta$ which intersects exactly one of $\alpha_v$ or $\alpha_w$ once, and is disjoint from the other, as illustrated in Figure~\ref{orb}.

 \begin{figure}[!ht]
    \centering
    \includegraphics[width=12 cm]{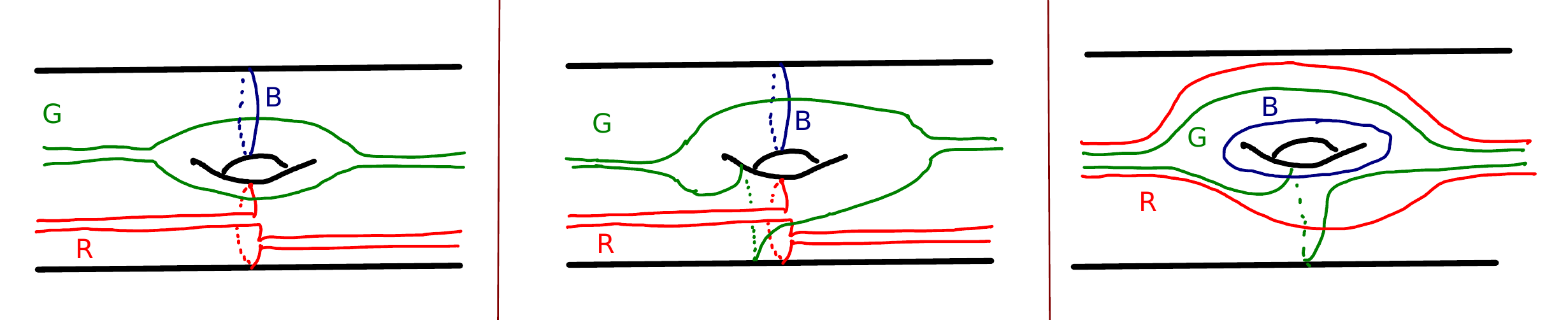}
    \caption{Illustrative examples of finding a new simple closed curve $B$ (in blue) disjoint from the red curve $R$ and having geometric intersection number one with the green curve $G$. }
    \label{orb}
\end{figure}
 \begin{rem}
 As Dan Margalit pointed out to us, the above construction can also be done using double branched covers, which can be more useful in certain situations. By quotienting out by the hyperelleptic involution, we can consider $\Sigma_g$ as a double branched cover over the sphere $\Sigma_0$, with $2g+2$ branch points, $y_1,...,y_{2g+2}$. We can think of the branched cover of the pair of branch points $y_{2g+1},y_{2g+2}$ as forming a tube connecting two disjoint $\Sigma_0^{g+1}$, and the rest of the pairs of correspond to adding genus.
 For each subset of the first $2g$ branch points, we can consider a simple closed curve in $\Sigma_{0,2g+2}$ enclosing these points (we think of the region containing $y_{2g+2}$ as outside), and if necessary $y_{2g+1}$ so that the total number of points is even. The lift of this curve realizes the mod two homology class of the binary vector corresponding to which branch points were chosen (in fact there is a bijection between $H_1(\Sigma_g;\mathbb{Z}_2)$ and the even subgroup of $H_1(\Sigma_{0,2g+2};\mathbb{Z}_2)$). Lastly, let us observe that given any two mod two non-homologous simple closed curves in $\Sigma_{0,2g+2}$, it is possible to choose an arc joining two branch points which intersects one and is disjoint from the other, and its lift is a simple closed curve in $\Sigma_g$ having the same property.
 \end{rem}
 \begin{lem}\label{lemB} For $g\geq 1$ and any non-cyclic quotient of $\mathrm{Mod}(\Sigma_{g})$, the $2^{2g}-1$ quotient classes $\overline{T_v}$ must be distinct.
 
 \end{lem}
\begin{proof}
Suppose we have two different binary vectors $v$ and $w$ so that $\overline{T_v}=\overline{T_w}$. By our above discussion, we can find a simple closed curve $\beta$, so that $T_\beta$ commutes with one of $\overline{T_v}$ or $\overline{T_w}$, and satisfies braid relation with the other. Hence, by the first observation in the alternate proof of Lemma~\ref{lemA}, we see that for two simple closed curves $c,d$ with geometric intersection number one, we have $\overline{T_c}=\overline{T_d}$. By \cite[Lemma 2.1]{Lanier2018NormalGF}, the quotient must be abelian (and hence cyclic since all abelianizations of $\mathrm{Mod}(\Sigma_{g})$ are cyclic~\cite[Chapter 5]{farb2012primer}), which gives a contradiction.
\end{proof}

\subsection*{Inductive step} In this step, we will consider the conjugation action on a pair of group elements, but the size of the orbit readily follows from the conjugation action considered in the first step.  

\begin{proof} [Proof of Theorem~\ref{C}] Let us first recall the statement we are going to prove.\\
  \noindent\textit{Inductive hypothesis:} Let $g\geq 1$. For any non-cyclic quotient $H$ of $\mathrm{Mod}(\Sigma_{g})$, either $|H|>|\mathrm{Sp}(2g,\mathbb{Z}_2)|$, or $H$ is isomorphic to $\mathrm{Sp}(2g,\mathbb{Z}_2)$. Moreover, in the latter case the quotient map $\mathrm{Mod}(\Sigma_{g})\rightarrow H$ is obtained by postcomposing $\Phi$ with an automorphism of $\mathrm{Sp}(2g,\mathbb{Z}_2)$.\\
We will use induction on $g$, and we note that the base case $g=1$ follows from Claim~\ref{c34}. We will inductively assume the statement is true for $k=g-1$ (with $g\geq 2$), and prove it for $k=g$.
Let $q:\mathrm{Mod}(\Sigma_{g})\rightarrow H$ be a quotient of smallest order.
Let $R$ and $S$ denote the right handed Dehn twists about the simple closed curves $\alpha_{e_1}$ and $\alpha_{e_2}$ (we use $e_i$ to denote the $i$-th standard basis vector in $\mathbb{Z}^{2g}$, and the same notation as in the previous section).
By Lemma \ref{lemB}, we know that the conjugacy class of the quotient class $\overline{R}$ in $H$ has size at least $2^{2g}-1$. We will consider the conjugation action of $G$ on the set of all ordered pairs of elements in $G$.
Since our original collection of curves $\alpha_v$, we have $2^{2g-1}(2^{2g}-1)$ ordered pairs with geometric intersection number one, by the change of coordinates principle~\cite[Section 1.3.3]{farb2012primer}, we see that the orbit of the ordered pair $(\overline{R},\overline{S})$ under the conjugation action is at least $2^{2g-1}(2^{2g}-1)$.
We see the stabilizer of ($\overline{R},\overline{S}$) contains the image $I$ under $q$ of $\mathrm{Mod}(\Sigma_{g-1}^1)$ (where $\Sigma_{g-1}^1$ is obtained by cutting $\Sigma_g$ along the separating curve which is the boundary of a regular neighbourhood of $\alpha_{e_1}$ and $\alpha_{e_2}$, i.e. we are deleting the leftmost genus containing $\alpha_{e_1}$ and $\alpha_{e_2}$), since $\mathrm{Mod}(\Sigma_{g}^1)$ fixes $\alpha_{e_1}$ and $\alpha_{e_2}$.
If $Z(I)$ denotes the center of this image $I$, we see that $I/Z(I)$ is a non-cyclic quotient (otherwise $I$ must be abelian, and thus the various conjugate $T_v$'s must map to the same element, contradicting Lemma~\ref{lemB}) of $\mathrm{Mod}(\Sigma_{g-1}^1)$. Since the boundary parallel Dehn twist in $\mathrm{Mod}(\Sigma_{g-1}^1)$ is central, it follows that $I/Z(I)$ is also a non-cyclic quotient of $\mathrm{Mod}(\Sigma_{g-1})$. By inductive hypothesis for $k=g-1$, we have that $|I|\geq |I/Z(I)| \geq|\mathrm{Sp}(2g-2,\mathbb{Z}_2)|$. Thus, by the orbit stabilizer theorem we have \begin{equation}
    |H|\geq 2^{2g-1}(2^{2g}-1)|I|\geq 2^{2g-1}(2^{2g}-1)|\mathrm{Sp}(2g-2,\mathbb{Z}_2)| =|\mathrm{Sp}(2g,\mathbb{Z}_2)|.
\end{equation}
Thus we get the desired result at the numerical level, and moreover in case of equality above we see that $Z(I)$ is trivial.
Moreover from the inductive hypothesis we have $I$ is isomorphic to $\mathrm{Sp}(2g-2,\mathbb{Z}_2)$. It follows that separating twists and for $g\geq 3$, genus one bounding pairs are in the kernel of $q$. Since by results of Birman, Powell and Johnson~\cite{hatcherm}, for $g\geq 3$ (respectively $g=2$) genus one bounding pairs (respectively separating twists) normally generate the Torelli group, we see that $q$ factors through $q_1:\mathrm{Sp}(2g,\mathbb{Z})\rightarrow  H$. Moreover by the inductive hypothesis some $\overline{T_v}$ has order 2, and so the kernel of $q_1$ contains squares of all transvections, and thus by \cite[Proposition A3]{tata1}, the kernel of $q_1$ contains the level two congruence subgroup. Consequently, $q$ in fact factors through $\mathrm{Sp}(2g,\mathbb{Z}_2)$, and the result follows.
\end{proof}

\section{Allowing punctures and boundary components}
In this final section, we will see some results about smallest non-cyclic/non-trivial quotients of $\mathrm{Mod}(\Sigma_{g,n}^b)$. These results are consequence of our main results and facts about the abelianizations of mapping class groups, discussed below.

\emph{Abelianization of Mapping Class Groups}: It is known~\cite[Theorem 5.1]{Korkmaz2002LowdimensionalHG} that 
the abelianization of the pure\footnote{We caution the reader that the reference we are citing follows the convention that mapping classes fix punctures and thus their mapping class group coincides with our pure mapping class group.} mapping class group $\mathrm{PMod}(\Sigma_{g,n})$ is:
\begin{enumerate}
    \item $\mathbb{Z}_{12}$ if $g=1$, $b=0$;
    \item $\mathbb{Z}^b$ if if $g=1$, $b\geq 1$;
    \item $\mathbb{Z}_{10}$ if $g=2$; and
    \item trivial if $g\geq 3$.
\end{enumerate}

This implies the following result (likely known, but we could not find it in the literature):
\begin{lem}\label{llab}
The abelianization of $\mathrm{Mod}(\Sigma_{g,n}^b)$ equals $\mathbb{Z}_{2}$ for $g\geq 3$ and $n\geq 2$.
\end{lem}

\begin{proof} By the above result, and the change of coordinates principle, we see under the abelianization map of $\mathrm{Mod}(\Sigma_{g,n}^b)$, all essential Dehn twists map to the identity, and all right handed half twists map to the same element. If we consider the subsurface $\Sigma_{g}^c$ of $\Sigma_{g,n}^b$ so that almost all the additional boundary components added consists of standard loops enclosing exactly two punctures (and one containing a single puncture if $n$ is odd), we see that squares of half twists must also map to the identity in the abelianization of $\mathrm{Mod}(\Sigma_{g,n}^b)$. The result follows by noting that the abelianization cannot be trivial since we have an epimorphism from $\mathrm{Mod}(\Sigma_{g,n}^b)$ to $S_n$, and hence to $\mathbb{Z}_2$. 
\end{proof}

We now find the smallest non-trivial quotient of $\mathrm{Mod}(\Sigma_{g,n}^b)$ for $g\geq 1$, and arbitrary $n, b$.
\begin{thm} \label{D}
The smallest non-trivial quotient of $\mathrm{Mod}(\Sigma_{g,n}^b)$ of smallest order is:
\begin{enumerate}
    \item $\mathbb{Z}_2$ for $n\geq 2$ or $g\in\{1,2\}$,  and arbitrary $b$,
    \item $\mathrm{Sp}(2g,\mathbb{Z}_2)$ for $g\geq 3$ and $n\in\{0,1\}$,  and arbitrary $b$.
\end{enumerate}

\end{thm}

\begin{proof}
For $n\geq 2$, we get an epimorphism $\mathrm{Mod}(\Sigma_{g,n}^b)\rightarrow S_n$ by considering the action on the punctures, and we can further quotient to the unique smallest non-trivial group $\mathbb{Z}_2$.
Thus it only remains to consider $n\in \{0,1\}$, and so all mapping classes are pure. From the aforementioned result about abelianization, we see that for $g\in\{1,2\}$ the smallest non-trivial quotient is $\mathbb{Z}_2$. Also, the same result tell us that for $g \geq 3$ there can be no non-trivial abelian quotients. Hence thus all boundary parallel and  puncture surrounding Dehn twists (which are central) must map to the identity under any non-trivial quotient of smallest order (otherwise we get an even smaller non-trivial quotient by quotienting by the center), and thus we reduce to the case in Theorem~\ref{C}.
\end{proof}

We also find the smallest non-cyclic quotient of  $\mathrm{Mod}(\Sigma_{g,n}^b)$ for a wide range of cases.
\begin{thm} \label{E} Any non-cyclic quotient of $\mathrm{Mod}(\Sigma_{g,n}^b)$ of smallest order is:
\begin{enumerate}
\item  the smaller of the groups among $S_n$ and $\mathrm{Sp}(2g,\mathbb{Z}_2)$ for $g\geq 3$, $n\geq 5$ and arbitrary $b$,

\item $S_3$ for  $g\geq 3$, $n\in \{3,4\}$ and arbitrary $b$,

\item $\mathrm{Sp}(2g,\mathbb{Z}_2)$ for $g\geq 2$,  $n\in \{0,1\}$, and arbitrary $b$ (also for $g=1$, and  $n,b\in \{0,1\}$),

\item $\mathbb{Z}_2\oplus \mathbb{Z}_2$ for $g=1$, $n\in \{0,1\}$, and $b\geq 2$.
\end{enumerate}

\end{thm}

\begin{proof}
Let us consider the center of a non-cyclic quotient of $\mathrm{Mod}(\Sigma_{g,n}^b)$ of smallest order. The only way this center is non-trivial is if the quotient is non-cyclic abelian (otherwise we get a strictly smaller non-cyclic quotient). This situation does happen for $g=1$, $n\in \{0,1\}$, and $b\geq 2$, where the abelianization of $\mathrm{Mod}(\Sigma_{g,n}^b)$ is $\mathbb{Z}^b$, which has the Klein four group (the unique non-cyclic group of smallest order) as a quotient.

Also, the above is the only case (among the ones mentioned in the statement) where this can happen, since the abelianization of $\mathrm{Mod}(\Sigma_{g,n}^b)$ is $\mathbb{Z}_2$ for $g\geq 3$ and $n\geq 2$, and $\mathbb{Z}/{10\mathbb{Z}}$ for $g=2$ and $n\in \{0,1\}$. Thus for these cases, the smallest non-cyclic quotient must necessarily be non abelian. Moreover, all boundary parallel Dehn twist must map to the trivial element in the quotient, and so we reduce to the case $b=0$ (and if $n=1$, the Dehn twist about the curve surrounding the puncture is also central, so we can also reduce to the case $n=0$). Hence for $g\geq 2$,  $n\in \{0,1\}$, and arbitrary $b$ (and also for $g=1$, and  $n,b\in \{0,1\}$) we reduce to the case $n=b=0$, and we get the desired result by Theorem ~\ref{C}.

 In case $n\in \{3,4\}$, and $g\geq 3$, we see that $S_3$ is a quotient of $\mathrm{Mod}(\Sigma_{g,n}^b)$ (using the induced action on the punctures and the exceptional homomorphism $S_4\rightarrow S_3$. As we saw earlier, $S_3$ must be the smallest quotient in these case as it is the unique smallest non-abelian group.

Finally, we now consider the case $g\geq 3$, $n\geq 5$ and $b=0$. Suppose  we have a quotient of $\mathrm{Mod}(\Sigma_{g,n})$, so that the restriction to both $\mathrm{Mod}(\Sigma_{g}^1)$ and $B_n\cong \mathrm{Mod}(\Sigma_{0,n}^1)$ are both cyclic. Then by Theorem~\ref{D} it must be the case that image of $\mathrm{Mod}(\Sigma_{g}^1)$ is trivial. Moreover by the braid relation, all half twists in $\mathrm{Mod}(\Sigma_{g,n})$ must map to a single element. Given any Dehn twist in $\mathrm{Mod}(\Sigma_{g,n})$, by a change of coordinate we can find a half twist commuting with it. So we see that the image of all half twists is a central element in the quotient, as $\mathrm{Mod}(\Sigma_{g,n})$ is generated by Dehn twists and half twists. This contradicts our observation earlier, so one of the restrictions to $\mathrm{Mod}(\Sigma_{g}^1)$ or $B_n$ is non-cyclic, giving us the desired result by using Theorems~\ref{A} and \ref{D}.\end{proof}

To complete the proof of Theorem~\ref{Cb}, it remains to verify the statement about maps. However, let us note that, the corresponding statement is not true for all the cases in Theorem~\ref{E}. For instance for $b\geq 3$ there are multiple epimorphisms from  $\mathrm{Mod}(\Sigma_{1}^b)$ to $\mathbb{Z}_2\oplus\mathbb{Z}_2$, even up to postcomposing by automorphisms of the image.

\begin{proof} [Proof of Theorem~\ref{Cb}]

For $g\geq 3$ and $n\geq 5$, let us first consider the case that the quotient of $\mathrm{Mod}(\Sigma_{g,n}^b)$ of smallest order is $S_n$. We know from the proof of Theorem~\ref{E} that
we can reduce to the case $b=0$, and the restriction of this quotient on  $\mathrm{Mod}(\Sigma_{0,n}^1)$ is $S_n$ as well. Since
$\mathrm{Mod}(\Sigma_{0,n}^1)$ commutes with $\mathrm{Mod}(\Sigma_{g}^1)$, and $S_n$ is centerless, it follows that $\mathrm{Mod}(\Sigma_{g}^1)$ is in the kernel of this quotient map. As all Dehn twists in $\mathrm{Mod}(\Sigma_{g,n})$ are conjugate, it follows that the kernel contains the pure mapping class group
$\mathrm{PMod}(\Sigma_{g,n})$. Consequently, the quotient map factors through $\mathrm{Mod}(\Sigma_{g,n})/\mathrm{PMod}(\Sigma_{g,n})\cong S_n$, and the desired result follows.

For $g\geq 3$ and $n\geq 5$, let us now consider the case that the quotient of  $\mathrm{Mod}(\Sigma_{g,n}^b)$ of smallest order is $\mathrm{Sp}(2g,\mathbb{Z}_2)$. Similar to our above discussion, we see that the quotient map restricted to $\mathrm{Mod}(\Sigma_{g}^1)$ is surjective, and all half twists are in the kernel of the quotient map. We know the epimorphism from $\mathrm{Mod}(\Sigma_{g}^1)$
to $\mathrm{Sp}(2g,\mathbb{Z}_2)$ has to send the boundary parallel Dehn twist to identity, and so it factors through $\mathrm{Mod}(\Sigma_{g})$. By Theorem~\ref{C}, we know this map is the standard projection $\Phi$, up to an automorphism $h$ of $\mathrm{Sp}(2g,\mathbb{Z}_2)$. By looking at the action on $\mathbb{Z}_2^{2g}$, we see that all the Dehn twists about curves not contained in $\Sigma_{g}^1$ (next to the punctures) in \cite[Figure 4.10]{farb2012primer} must map to the same element. Since we also know that all half twists are in the kernel of the quotient map $\mathrm{Mod}(\Sigma_{g,n}^b)\rightarrow \mathrm{Sp}(2g,\mathbb{Z}_2) $, it follows that this map coincides with the standard projection, postcomposed with the same automorphism $h$ of $\mathrm{Sp}(2g,\mathbb{Z}_2)$.

For $g\geq 3$ and $n\in\{0,1\}$, the result follows by the same argument in the last paragraph. Lastly, for $g\geq 3$ and $n\geq 2$, any homomorphism from $\mathrm{Mod}(\Sigma_{g,n}^b)$ to an abelian group must factor through the abelianization of $\mathrm{Mod}(\Sigma_{1}^b)$, which by Lemma~\ref{llab} is $\mathbb{Z}_2$. Hence the result follows, and moreover this map is unique since $\mathbb{Z}_2$ does not have a non-trivial automorphism.\end{proof}

\bibliographystyle{plain}
\bibliography{references}

\end{document}